\documentclass[11pt,reqno]{amsart}

\usepackage{graphicx}
\usepackage[table]{xcolor}
\usepackage{xspace,tikz}
\usepackage{amsmath,amsthm,amsfonts,amssymb,latexsym,mathrsfs,color,extarrows}
\usepackage{hyperref}

\newtheorem{theorem}{Theorem}
\newtheorem{corollary}[theorem]{Corollary}
\newtheorem{proposition}[theorem]{Proposition}
\newtheorem{conjecture}[theorem]{Conjecture}

\newtheorem{lemma}[theorem]{Lemma}
\newtheorem{definition}[theorem]{Definition}
\newtheorem{example}[theorem]{Example}

\newtheorem{problem}[theorem]{Problem}
\newtheorem*{HBtheorem}{Hermite-Biehler Theorem}
\newtheorem*{MFSaction}{MFS-action}

\newcommand{\maj}{{\rm maj\,}}
\newcommand{\MFS}{{\rm MFS\,}}

\newcommand{\val}{{\rm val\,}}

\newcommand{\UB}{{\rm UB\,}}

\newcommand{\fap}{{\rm fap\,}}
\newcommand{\dasc}{{\rm dasc\,}}

\newcommand{\ap}{{\rm ap\,}}
\newcommand{\lap}{{\rm lap\,}}
\newcommand{\lpk}{{\rm lpk\,}}

\newcommand{\pk}{{\rm pk\,}}
\newcommand{\ddes}{{\rm ddes\,}}
\newcommand{\des}{{\rm des\,}}

\newcommand{\exc}{{\rm exc\,}}

\newcommand{\msn}{\mathfrak{S}_n}

\newcommand{\ms}{\mathfrak{S}}

\newcommand{\lrf}[1]{\lfloor #1\rfloor}
\newcommand{\lrc}[1]{\lceil #1\rceil}
\newcommand{\sgn}{{\rm sgn\,}}

\newcommand{\mqn}{\mathcal{Q}_n}
\newcommand{\z}{ \mathbb{Z}}
\newcommand{\asc}{{\rm asc\,}}

\DeclareMathOperator{\N}{\mathbb{N}}

\DeclareMathOperator{\R}{\mathbb{R}}

\newcommand{\rz}{{\rm RZ}}
\newcommand{\Eulerian}[2]{\genfrac{<}{>}{0pt}{}{#1}{#2}}

\title{Positivity of Narayana polynomials and Eulerian polynomials}
\author[S.-M.~Ma]{Shi-Mei Ma}
\address{School of Mathematics and Statistics,
        Northeastern University at Qinhuangdao,
         Hebei 066000, P.R. China}
\email{shimeimapapers@163.com (S.-M. Ma)}
\author{Hao Qi}
\address{College of mathematics and physics, Wenzhou University, Wenzhou 325035, P.R. China}
\email{qihao@wzu.edu.cn (H. Qi)}
\author{Jean Yeh}
\address{Department of Mathematics, National Kaohsiung Normal University, Kaohsiung 82446, Taiwan}
\email{chunchenyeh@nknu.edu.tw (J. Yeh)}
\author{Yeong-Nan Yeh}
\address{Institute of Mathematics, Academia Sinica, Taipei, Taiwan}
\email{mayeh@math.sinica.edu.tw (Y.-N. Yeh)}
\begin{document}
\maketitle
\begin{abstract}
Gamma-positivity appears frequently in finite geometries, combinatorics and number theory.
Motivated by the recent work of Sagan and Tirrell (Adv.~Math., 374 (2020), 107387),
we study the relationships between gamma-positivity and alternating gamma-positivity.
As applications, we derive several alternatingly gamma-positive polynomials related to
Narayana polynomials and Eulerian polynomials.
In particular, we show the alternating gamma-positivity and Hurwitz stability of a combination of the modified Narayana polynomials of types $A$ and $B$.
By using colored $2\times n$ Young diagrams,
we present a unified combinatorial interpretations of three identities involving Narayana numbers of type $B$.
A general result of this paper is that every gamma-positive polynomial is also alternatingly semi-gamma-positive.
At the end of this paper, we pose two conjectures, one concerns the Boros-Moll polynomials and the other concerns the enumerators of permutations by descents and excedances.
\bigskip

\end{abstract}
\date{\today}
\newpage
\tableofcontents


\section{Introduction}
Let $\Delta$ be a simplicial complex of dimension $n-1$. The {\it $f$-vector} of $\Delta$ is the sequence of integers
$(f_{-1},f_0,f_1,\ldots,f_{n-1})$, where $f_i$ is the number of faces with $i+1$ vertices in $\Delta$.
For example, $f_{-1}=1$, corresponding to the empty face.
The $f$-polynomial and $h$-polynomial of $\Delta$ are respectively defined as
$f(x)=\sum_{i=0}^nf_{i-1}x^i$, and
$$h(x)=(1-x)^nf\left(\frac{x}{1-x}\right)=\sum_{i=0}^nf_{i-1}x^i(1-x)^{n-i}=\sum_{i=0}^nh_ix^i.$$
The sequence of integers $(h_0,h_1,\ldots,h_n)$ is called the {\it $h$-vector} of $\Delta$.
It is known that the $h$-polynomial of a simple
polytope is positive and symmetric~\cite{Postnikov08}.
In~\cite{Fomin03}, Fomin and Zelevinsky defined the (generalized) Narayana numbers $N_k(\Phi)$
for an arbitrary root system $\Phi$ as the entries of the $h$-vector of the simplicial complex
dual to the corresponding generalized associahedron.
Let $N(\Phi,x)=\sum_{k=0}^nN_k(\Phi)x^k$. For the classical Weyl groups, the generating polynomials for the Narayana numbers are given as follows:
\begin{equation*}\label{AnxBnx}
\begin{split}
N(A_n,x)&=\sum_{k=0}^n\frac{1}{n+1}\binom{n+1}{k+1}\binom{n+1}{k}x^k,\\
N(B_n,x)&=\sum_{k=0}^n{\binom{n}{k}}^2x^k,\\
N(D_n,x)&=N(B_n,x)-nxN(A_{n-2},x),
\end{split}
\end{equation*}
where $A_n$ is group of permutations of $\{1,2,\ldots,n+1\}$, $B_n$ is
the group of signed permutations of $\{\pm1,\pm2,\ldots,\pm n\}$ and $D_n$ is the group of even-signed permutations in $B_n$.
Narayana polynomials possess many of the same properties as Eulerian polynomials,
including real-rootedness, symmetry property and $\gamma$-positivity,
and there are several combinatorial and geometric interpretations.
For example, ${N}_n(A_n,x)$ is the enumerator of 231-avoiding permutations in $\ms_{n+1}$ by descents, and
${N}_n(B_n,x)$ is the enumerator of $(1342,3142,3412,3421)$-avoiding permutations in $\ms_{n+1}$ by descents, see~\cite{Chen2008,Petersen15,Reiner97} for details.

Given a Coxeter system $(W,S)$ and $\sigma\in W$, we denote by $\ell_W(\sigma)$ the length of $\sigma$ in $W$.
The number of {\it $W$-descents} of $\sigma$ is defined by
$d_W(\sigma)=\#\{s\in S: \ell_W(\sigma s)<\ell_W(\sigma)\}$.
The Eulerian polynomial of a finite Coxeter group $W$ is
$$P(W,x)=\sum_{\sigma\in W}x^{d_W(\sigma)}.$$
This polynomial is also the $h$-polynomial of the Coxeter complex associated to $(W, S)$.
For Coxeter groups of types $A_n$ and $B_n$, one has $P(A_n,x)=A_{n+1}(x)$ and $P(B_n,x)=B_n(x)$.
The {\it types $A$ and $B$ Eulerian polynomials} respectively satisfy the following recurrence relations:
\begin{equation}\label{AnxBnx}
\begin{split}
A_{n}(x)&=(nx+1-x)A_{n-1}(x)+x(1-x)\frac{\mathrm{d}}{\mathrm{d}x}A_{n-1}(x),\\
B_{n}(x)&=(2nx+1-x)B_{n-1}(x)+2x(1-x)\frac{\mathrm{d}}{\mathrm{d}x}B_{n-1}(x),
\end{split}
\end{equation}
with the initial conditions $A_0(x)=B_0(x)=1$ (see~\cite{Bre94,Chow08,Zhuang17} for instance).
Let $D=\frac{\mathrm{d}}{\mathrm{d}x}$ be the differential operator.
The Eulerian polynomials first appearance in series
summation or successive differentiation:
\begin{equation*}\label{Anxsum}
\left(xD\right)^n\frac{1}{1-x}=\sum_{k=0}^\infty k^nx^k=\frac{xA_n(x)}{(1-x)^{n+1}}.
\end{equation*}
Using~\eqref{AnxBnx}, one can easily verify that
$$\left(xD\right)^n\frac{1}{1-x^2}=\frac{2^nx^2A_n(x^2)}{(1-x^2)^{n+1}},~\left(xD\right)^n\frac{x}{1-x^2}=\frac{xB_n(x^2)}{(1-x^2)^{n+1}}.$$
Since $\left(xD\right)^n \frac{1}{1-x}=\left(xD\right)^n\frac{1}{1-x^2}+\left(xD\right)^n\frac{x}{1-x^2}$, one has
\begin{equation}\label{ANxBNx}
(1+x)^{n+1}A_n(x)=B_n(x^2)+2^nxA_n(x^2),
\end{equation}
which frequently appeared in literatures, see~\cite[Theorem~3]{Ma121} for instance.

As usual, we use $\msn$ to denote the symmetric group of all permutations of $[n]=\{1,2,\ldots,n\}$. Let $\pi=\pi(1)\pi(2)\cdots\pi(n)\in\msn$.
In this paper, we always assume that $\pi(0)=\pi(n+1)=\infty$ (except where explicitly stated).
If $i\in [n]$, then $\pi(i)$ is called
\begin{itemize}
   \item [$\bullet$] a {\it descent} if $\pi(i)>\pi(i+1)$;
   \item [$\bullet$]  an {\it ascent} if $\pi(i)<\pi(i+1)$;
  \item [$\bullet$] a {\it peak} if $\pi(i-1)<\pi(i)>\pi(i+1)$;
  \item [$\bullet$] a {\it valley} if $\pi(i-1)>\pi(i)<\pi(i+1)$,
  \item [$\bullet$] a {\it double descent} if $\pi(i-1)>\pi(i)>\pi(i+1)$;
   \item [$\bullet$] a {\it double ascent} if $\pi(i-1)<\pi(i)<\pi(i+1)$.
\end{itemize}
Let $\des(\pi)$ (resp.~$\asc(\pi)$,~$\pk(\pi)$,~$\val(\pi)$,~$\ddes(\pi)$,~$\dasc(\pi)$) be
the number of descents (resp.~ascents,~peaks,~valleys,~double descents,~double ascent) of $\pi$.
The following expansion of the Eulerian polynomials $A_n(x)$ was first observed by Foata and Sch\"utzenberger~\cite{Foata70}:
\begin{equation}\label{Anx-gamma-Foata}
A_n(x)=\sum_{k=0}^{\lrf{(n-1)/2}}\gamma_{n,k}x^k(1+x)^{n-1-2k}~{\text{for $n\geqslant 1$},}
\end{equation}
where $\gamma_{n,k}=\#\{\pi\in\msn:~\pk(\pi)=k,~\ddes(\pi)=0\}$.
Recently there has been considerable interest in the refinements and generalizations of~\eqref{Anx-gamma-Foata},
see~\cite{Athanasiadis17,Ma19,Petersen15,Zhuang17} and references therein.

Assume that $f(x)=\sum_{i=0}^nf_ix^i$ is a symmetric polynomial of degree $n$, i.e.,
$f_i=f_{n-i}$ for any $0\leqslant i\leqslant n$. Then $f(x)$ can be expanded uniquely as
$$f(x)=\sum_{k=0}^{\lrf{{n}/{2}}}\gamma_kx^k(1+x)^{n-2k}.$$
We call $\{\gamma_k\}_{k=0}^{\lrf{n/2}}$ the {\it $\gamma$-vector} of $f(x)$. If $\gamma_k\geqslant 0$ for $0\leqslant k\leqslant \lrf{{n}/{2}}$,
then $f(x)$ is said to be {\it $\gamma$-positive} (see~\cite{Gal05}).
Notably, $\gamma$-positivity of a polynomial implies that its coefficients are symmetric and unimodal, and the
coefficients of $\gamma$-positive polynomials often have nice geometric and combinatorial interpretations, see~\cite{Athanasiadis17,Petersen15} for details.
If the $\gamma$-vector of $f(x)$ alternates in sign, then we say that $f(x)$ is {\it alternatingly $\gamma$-positive} (see~\cite{Brittenham16,Lin2011,Sagan20} for instance).
For example, $(1+x^2)^n$ is alternatingly $\gamma$-positive, since
$$(1+x^2)^n=[(1+x)^2-2x]^n=\sum_{k=0}^n\binom{n}{k}2^k(-x)^k(1+x)^{2n-2k},$$
where the alternating $\gamma$-coefficients $\binom{n}{k}2^k$ count $k$-simplices in the $n$-cube (see~\cite[A013609]{Sloane}).
There has been considerable recent interest in the study alternatingly $\gamma$-positive polynomials,
see~\cite{Brittenham16,Lin2011,Ma2201,Sagan20} for instance. For example,
Lin etal.~\cite{Lin2011} showed the alternating $\gamma$-positivity of alternating Eulerian polynomials.

The {\it Lucas polynomials} $\{n\}:=\{n\}_{s,t}$ are defined by $\{n\}=s\{n-1\}+t\{n-2\}$ with the initial conditions
$\{0\}=0,~\{1\}=1$.
When $s=1+q,t=-q$, one has
\begin{equation}\label{Lucas}
\{n\}=1+q+q^2+\cdots+q^{n-1}.
\end{equation}
The reader is referred to~\eqref{pnqn02} for the alternating $\gamma$-expansion of $\{n+1\}$.
Sagan and Tirrell~\cite{Sagan20} introduced a sequence of polynomials $P_n(s,t)$ by using the factorization of $\{n\}$:
$$\{n\}=\Pi_{d\mid n}P_d(s,t).$$ The polynomials $P_n(s,t)$ are called {\it Lucas atoms}.
Motivated by~\eqref{Lucas}, Sagan and Tirrell~\cite{Sagan20} first established a connection between cyclotomic
polynomials and Lucas atoms, and then proved the alternating $\gamma$-positivity of cyclotomic
polynomials. They also wrote in their paper~\cite[p.~24]{Sagan20}:``it might also be interesting to look at gamma expansions where the coefficients alternate in sign. Very little work has been done in this direction''.
Motivated by the work of Sagan and Tirrell~\cite{Sagan20},
it is natural to consider the following problem.
\begin{problem}\label{Problem1}
Are there some connections between $\gamma$-positivity and alternating $\gamma$-positivity?
\end{problem}

In this paper, we present various results concerning Problem~\ref{Problem1}.
Among other things, in Section~\ref{Section06},
we show that every $\gamma$-positive polynomial is also alternatingly semi-$\gamma$-positive. Moreover,
in Section~\ref{Section07}, we present two conjectures, one concerns the Boros-Moll polynomials, the other concerns
the enumerators of permutations by descents and excedances.
The main results of this paper are Theorems~\ref{thm01},~\ref{Nn},~\ref{thmMN},~\ref{thmEulerian01} and~\ref{Fnx}.
\section{Properties of the modified Foata-Strehl action}\label{Section02}
In~\cite{Branden08} Br\"and\'en introduced the following modified Foata-Strehl action ($\MFS$-action for short), which can be used to show the $\gamma$-positivity of various enumerative polynomials.
\begin{MFSaction}[{\cite{Branden08}}]\label{MFS}
Let $\pi\in\msn$ and $x=\pi(i)$. The {\it modified Foata-Strehl action} $\varphi_{x}$ on $\msn$ is defined as follows:
\begin{itemize}
  \item [$(i)$] If $x$ is a double descent, then $\varphi_{x}$
  is obtained by deleting $x$ and then inserting $x$ between $\pi(j)$ and $\pi(j+1)$, where $j$ is
  the smallest index satisfying $i<j$ and $\pi(j)<x<\pi(j+1)$;
  \item [$(ii)$] If $x$ is a double ascent, then $\varphi_{x}$
  is obtained by deleting $x$ and then inserting $x$ between $\pi(j)$ and $\pi(j+1)$, where $j$ is
  the largest index satisfying $j<i$ and $\pi(j)>x>\pi(j+1)$;
  \item [$(iii)$] If $x$ is a peak or a valley, then let $\varphi_{x}(\pi)=\pi$.
\end{itemize}
\end{MFSaction}
Let $\operatorname{Orb}(\pi)=\{g(\pi): g\in \z_2^n\}$ be the orbit of $\pi$ under the $\MFS$-action.
Br\"and\'en noted that the following result follows from the work in~\cite{Foata74}, and he proved it by using the $\MFS$-action.
\begin{proposition}[{\cite[Theorem~3.1]{Branden08}}]
For any $\pi\in\msn$, one has
\begin{equation}\label{altgamma00}
\sum_{\sigma\in \operatorname{Orb}(\pi)}x^{\des(\sigma)}=x^{\des(\widehat{\pi})}(1+x)^{n-1-2\des(\widehat{\pi})}=x^{\pk(\pi)}(1+x)^{n-1-2\pk(\pi)},
\end{equation}
where $\widehat{\pi}$ to denote the unique element in $\operatorname{Orb}(\pi)$ with no double descents.
\end{proposition}
An immediate consequence of~\eqref{altgamma00} is the following identity:
\begin{equation}\label{altgamma01}
\sum_{\sigma\in \operatorname{Orb}(\pi)}x^{2\des(\sigma)}=\sum_{i=0}^{n-1-2\pk({\pi})}\binom{n-1-2\pk({\pi})}{i}2^i(-x)^{2\pk({\pi})+i}(1+x)^{2n-2-2(2\pk({\pi})+i)}.
\end{equation}
In the following, we shall give a combinatorial proof of~\eqref{altgamma01}. As illustrated in subsection~\ref{SunsectionNa}, along the same lines,
one can derive alternating $\gamma$-expansions of various enumerative polynomials.

Let $\pi\in\msn$. We can draw a permutation as a mountain range such that
peaks and valleys form the upper and lower limits of decreasing runs. Since we set
$\pi(0)=\pi(n+1)=\infty$, we put points at infinity on the far left and far right.
We say that
\begin{itemize}
   \item [$\bullet$] $\pi(i)\pi(i+1)$ is a {\it descent segment} if $\pi(i)>\pi(i+1)$, where $i\in [n-1]$;
   \item [$\bullet$] $\pi(i)\pi(i+1)$ is an {\it ascent segment} if $\pi(i)<\pi(i+1)$, where $i\in [n-1]$;
   \item [$\bullet$] $\pi(i)\pi(i+1)$ is a {\it double descent segment} if $\pi(i-1)>\pi(i)>\pi(i+1)$, where $i\in [n-1]$;
   \item [$\bullet$] $\pi(i-1)\pi(i)$ is a {\it double ascent segment} if $\pi(i-1)<\pi(i)<\pi(i+1)$, where $i\in [n]$;
  \item [$\bullet$] $\pi(i-1)\pi(i)$ an {\it ascent segment of peak} if $\pi(i-1)<\pi(i)>\pi(i+1)$, where $i\in [n-1]$;
   \item [$\bullet$] $\pi(i)\pi(i+1)$ a {\it descent segment of peak} if $\pi(i-1)<\pi(i)>\pi(i+1)$, where $i\in [n-1]$.
\end{itemize}
We use $U$ and $D$ to denote an ascent segment of peak and a descent segment of peak, respectively.
For any $c\in\N$, we define a {\it $c$-colored permutation} to be a permutation with the double descent segments must be colored by one of $c$ colors.
When $c=1$, the $c$-colored permutations reduce to ordinary permutations.
When $c=3$, a double descent segment may be colored by $B,R$ or $G$, we use $B,R$ and $G$ to
stand for a blue segment, a red segment and a green segment, respectively.
Let $\operatorname{Corb}(\pi)$ be the set of $3$-colored permutations generated by all the permutations in $\operatorname{Orb}(\pi)$.
As usual, the weight of a permutation is the product of the weights of its values,
and the weight of a set of permutations is the sum of the weights of permutations.

\noindent{\bf Combinatorial interpretation of the identity~\eqref{altgamma01}:}\\
For any $\sigma\in \operatorname{Orb}(\pi)$, we assign the weight $x^2$ to each descent segment and the weight $1$ to each ascent segment.
Then the left-hand side of~\eqref{altgamma01} equals the weight of $\operatorname{Orb}(\pi)$.

For any permutation in $\operatorname{Corb}(\pi)$,
we first assign the weight $x^2$ to each of the $D$ segments and the weight $1$ to each of the ascent segments.
We then reassign the weights of double descent segments:
assign $x^2$ to each of the $B$ segments, the weight $2x$ to each of the $R$ segments, the weight $-2x$ to each of the $G$ segments.
Since $x^2=x^2+2x-2x$, the weight of $\operatorname{Corb}(\pi)$ equals that of $\operatorname{Orb}(\pi)$.
Let $O(i)$ the subset of $\operatorname{Corb}(\pi)$, where each permutation has $i$ $G$ segments and has the $G$ segments in given positions.
By using the $\MFS$-action, the weight of $O(i)$ equals
$$(x^2)^{\pk(\pi)}(-2x)^i(1+x^2+2x)^{n-1-2\pk(\pi)-i}=2^i(-x)^{2\pk(\pi)+i}\left((1+x)^2\right)^{n-1-2\pk(\pi)-i}.$$
Since there are $\binom{n-1-2\pk({\pi})}{i}$ ways to choose $G$ segments, the weight of $\operatorname{Corb}(\pi)$
equals the right-hand side of~\eqref{altgamma01}. \qed
\section{Gamma-positivity and alternating gamma-positivity}\label{Section03}
Let $f(x)=\sum_{i=0}^nf_ix^i\in\R[x]$.
Consider the linear operator $\mathcal{A}_m: \R[x]\rightarrow \R[x]$ defined by
$\mathcal{A}_m(f(x))=f(x^{m})$.
The operator $\mathcal{A}_m$ appears frequently in
the study of field theory, number theory and polynomials (see~\cite{Gantmacher60,Wagner96,Ma22,Roberts07}).
For example, Roberts~\cite{Roberts07} studied fractalized cyclotomic polynomials by using $\mathcal{A}_m$.
As illustrations, there are two reduction formulas of cyclotomic polynomials (see~\cite{Roberts07,Sagan20}):
$$\Phi_p(x)=1+x+x^2+\cdots+x^{p-1},~\Phi_{pn}(x)=\frac{\Phi_n(x^p)}{\Phi_n(x)},$$
where $n\in\N$ and $p$ is a prime not dividing $n$.
Moreover, the {\it Hermite-Biehler decomposition} of $f(x)$ is given as follows:
$$f(x)=\sum_{k=0}^{\lrf{{n}/{2}}}f_{2k}x^{2k}+x\sum_{k=0}^{\lrf{{(n-1)}/{2}}}f_{2k+1}x^{2k}=f^E(x^2)+xf^O(x^2)=\mathcal{A}_2f^E(x)+x\mathcal{A}_2f^O(x).$$

Let us now recall two well known formulas:
\begin{equation}\label{pnqn}
p^n+q^n=\sum_{k=0}^{\lrf{n/2}}(-1)^k\frac{n}{n-k}\binom{n-k}{k}(pq)^k(p+q)^{n-2k};
\end{equation}
\begin{equation}\label{pnqn02}
\sum_{i=0}^np^iq^{n-i}=\sum_{k=0}^{\lrf{n/2}}(-1)^k\binom{n-k}{k}(pq)^k(p+q)^{n-2k}.
\end{equation}
In the past decades, these formulas frequently appeared in combinatorics and number theory,
see~\cite[Section~3]{Brittenham16},~\cite[p.~156]{Comtet74}),~\cite[p.~1068]{Guo06},~\cite[Example~6.11]{Postnikov08} for instance.
Based on the structure of matchings on path and cycle graphs, Brittenham, Carroll, Petersen, and Thomas~\cite{Brittenham16} provided combinatorial interpretations for the
alternating $\gamma$-expansions of $1+q^n,~\sum_{i=0}^nq^{i}$ and the $q$-binomial coefficients.

The following simple result will be used in our discussion.
\begin{lemma}\label{Lemmaalternate}
The product of two alternatingly $\gamma$-positive polynomials is alternatingly $\gamma$-positive.
\end{lemma}
\begin{proof}
Let $f(x)$ and $g(x)$ be two alternatingly $\gamma$-positive polynomials. Suppose that
$$f(x)=\sum_{k=0}^{\lrf{{n}/{2}}}\gamma_k(-x)^k(1+x)^{n-2k},~g(x)=\sum_{\ell=0}^{\lrf{{m}/{2}}}\eta_{\ell}(-x)^{\ell}(1+x)^{m-2\ell},$$
where $\gamma_k\geqslant 0$ for $0\leqslant k\leqslant \lrf{n/2}$ and $\eta_{\ell}\geqslant 0$ for $0\leqslant \ell\leqslant \lrf{m/2}$.
Then $$f(x)g(x)=\sum_{i=0}^{\lrf{(n+m)/2}}\sum_{k=0}^i\gamma_k\eta_{i-k}(-x)^{i}(1+x)^{n+m-2i},$$
as desired. This completes the proof.
\end{proof}
As a partial answer to Problem~\ref{Problem1}, we present the following result.
\begin{theorem}\label{thm01}
Assume that $f(x)=\sum_{i=0}^nf_ix^{i}=\sum_{i=0}^{\lrf{{n}/{2}}}\gamma_ix^i(1+x)^{n-2i}$.
\begin{itemize}
  \item [$(i)$] If $f(x)$ is $\gamma$-positive, then $\mathcal{A}_{2m}(f(x))=f(x^{2m})$ is alternatingly $\gamma$-positive, where $m\in\N$.
  \item [$(ii)$]
We have
\begin{equation}\label{fx2}
f(x^2)=\sum_{k=0}^n\eta_k(-x)^k(1+x)^{2n-2k},
\end{equation}
where $\eta_k=\sum_{i=0}^{\lrf{{k}/{2}}}\binom{n-2i}{k-2i}2^{k-2i}\gamma_i$.
Since $x^2=(-x)^2$, an equivalent formula of~\eqref{fx2} is given as follows:
\begin{equation*}
f(x^2)=\sum_{k=0}^n\eta_kx^k(1-x)^{2n-2k},
\end{equation*}
 Moreover, the following two identities are equivalent:
  $$\sum_{i=0}^nf_ix^{2i}=\sum_{k=0}^n\eta_k(-x)^k(1+x)^{2n-2k},$$
  $$\sum_{i=0}^nf_ix^{2i}(1+x)^{2n-2i}=\sum_{k=0}^n\eta_kx^k(1+x)^k;$$
 \item [$(iii)$] Setting $\xi_k=\sum_{i=0}^{\lrf{{k}/{2}}}\binom{n-2i}{k-2i}\gamma_i$, we have two equivalent identities:
 \begin{equation}\label{E01}
\sum_{k=0}^n\eta_kx^k=\sum_{i=0}^{\lrf{{n}/{2}}}\gamma_ix^{2i}(1+2x)^{n-2i},
 \end{equation}
  \begin{equation}\label{E02}
\sum_{k=0}^n\eta_kx^k=\sum_{k=0}^n\xi_kx^k(1+x)^{n-k}.
\end{equation}
 \item [$(iv)$] The modified $\gamma$-coefficient polynomial of $f(x)$ has two equivalent expansions:
 \begin{equation}\label{gammaxi}
 \sum_{i=0}^{\lrf{{n}/{2}}}\gamma_ix^{2i}=\sum_{k=0}^n\xi_k(-x)^k(1+x)^{n-k}=\sum_{k=0}^n\xi_kx^k(1-x)^{n-k}.
 \end{equation}
\end{itemize}
\end{theorem}
\begin{proof}
\quad $(i)$
By using~\eqref{pnqn}, we get
$1+x^{2m}=\sum_{j=0}^{m}a_{2m,j}(-x)^j(1+x)^{2m-2j}$,
where $a_{m,j}=\frac{m}{m-j}\binom{m-j}{j}$.
This yields
$$f(x^{2m})=\sum_{i=0}^{\lrf{{n}/{2}}}\gamma_ix^{2mi}[1+x^{2m}]^{n-2i}=\sum_{i=0}^{\lrf{{n}/{2}}}\gamma_ix^{2mi}\left\{\sum_{j=0}^{m}a_{2m,j}(-x)^j(1+x)^{2m-2j}\right\}^{n-2i}.$$
By using Lemma~\ref{Lemmaalternate}, we obtain
\begin{equation*}
\begin{split}
f(x^{2m})&=\sum_{i=0}^{\lrf{{n}/{2}}}\gamma_ix^{2mi}\sum_{\ell=0}^{m(n-2i)}b_{m,\ell}(-x)^{\ell}(1+x)^{2mn-4mi-2\ell}\\
&=\sum_{i=0}^{\lrf{{n}/{2}}}\sum_{\ell=0}^{m(n-2i)}\gamma_ib_{m,\ell}(-x)^{2mi+\ell}(1+x)^{2mn-2(2mi+\ell)}\\
&=\sum_{k=0}^{mn}\left(\sum_{2mi+\ell=k}\gamma_ib_{m,\ell}\right)(-x)^k(1+x)^{2mn-2k},
\end{split}
\end{equation*}

where $b_{m,\ell}=\sum_{(i_0,i_1,i_2,\ldots,i_m)}\frac{(n-2i)!}{i_0!i_1!i_2!\cdots i_m!}a^{i_0}_{2m,0}a_{2m,1}^{i_1}a_{2m,2}^{i_2}\cdots a_{2m,m}^{i_m}$,
and the summation is over all sequences of nonnegative integers $(i_0,\ldots,i_m)$ such that
$\sum_{j=1}^mji_j=\ell$ and $\sum_{j=0}^m i_j=n-2i$.
Thus $f(x^{2m})$ is alternatingly $\gamma$-positive.

\quad $(ii)$
Note that
\begin{equation*}
\begin{split}
f(x^2)&=\sum_{i=0}^{\lrf{{n}/{2}}}\gamma_ix^{2i}[(1+x)^2-2x]^{n-2i}\\
&=\sum_{i=0}^{\lrf{{n}/{2}}}\sum_{\ell=0}^{n-2i}2^{\ell}\binom{n-2i}{\ell}\gamma_i(-x)^{2i+\ell}(1+x)^{2n-2(2i+\ell)}\\
&=\sum_{k=0}^n\sum_{i=0}^{\lrf{{k}/{2}}}\binom{n-2i}{k-2i}2^{k-2i}\gamma_i(-x)^k(1+x)^{2n-2k}.
\end{split}
\end{equation*}
Then we obtain
\begin{equation*}
\begin{split}
(1+x)^{2n}\sum_{i=0}^nf_i{\left(\frac{x}{1+x}\right)}^{2i}&=(1+x)^{2n}\sum_{i=0}^{\lrf{{n}/{2}}}\gamma_i{\left(\frac{x}{1+x}\right)}^{2i}{\left(1+\frac{x^2}{(1+x)^2}\right)}^{n-2i}\\
&=\sum_{i=0}^{\lrf{{n}/{2}}}\gamma_i(x(1+x))^{2i}(1+2x(1+x))^{n-2i}\\
&=\sum_{i=0}^{\lrf{{n}/{2}}}\sum_{\ell=0}^{n-2i}\binom{n-2i}{\ell}2^{\ell}\gamma_i(x(1+x))^{2i+\ell}\\
&=\sum_{k=0}^n\left\{\sum_{i=0}^{\lrf{{k}/{2}}}\binom{n-2i}{k-2i}2^{k-2i}\gamma_i\right\}x^k(1+x)^{k}\\
&=\sum_{k=0}^n\eta_kx^k(1+x)^{k},
\end{split}
\end{equation*}
and vice versa.

\quad $(iii)$ On the one hand, we have
$$\sum_{i=0}^{\lrf{{n}/{2}}}\gamma_ix^{2i}(1+2x)^{n-2i}=\sum_{i=0}^{\lrf{{n}/{2}}}\sum_{\ell=0}^{n-2i}\binom{n-2i}{\ell}2^{\ell}\gamma_ix^{2i+\ell}=\sum_{k=0}^n\eta_kx^k.$$
On the other hand,
\begin{equation*}
\begin{split}
\sum_{k=0}^n\eta_kx^k&=\sum_{i=0}^{\lrf{{n}/{2}}}\gamma_ix^{2i}(1+x+x)^{n-2i}\\
&=\sum_{i=0}^{\lrf{{n}/{2}}}\sum_{\ell=0}^{n-2i}\binom{n-2i}{\ell}\gamma_ix^{2i+\ell}(1+x)^{n-2i-\ell}\\
&=\sum_{k=0}^n\xi_kx^k(1+x)^{n-k},
\end{split}
\end{equation*}
as desired. Thus $\sum_{k=0}^n\eta_kx^k$ has two equivalent expansions.

\quad $(iv)$ Making the
substitution $\frac{x}{1+2x}=y$, it follows from~\eqref{E01} and~\eqref{E02} that
\begin{equation*}
\begin{split}
\sum_{i=0}^{\lrf{n/2}}\gamma_iy^{2i}&=\sum_{k=0}^n\eta_ky^k(1-2y)^{n-k}\\
&=(1-2y)^n\sum_{k=0}^n\xi_k\left(\frac{y}{1-2y}\right)^k\left(\frac{1-y}{1-2y}\right)^{n-k}\\
&=\sum_{k=0}^n\xi_ky^k(1-y)^{n-k}.
\end{split}
\end{equation*}
Since $y^2=(-y)^2$, it follows that
$$\sum_{i=0}^{\lrf{n/2}}\gamma_iy^{2i}=\sum_{k=0}^n\xi_k(-y)^k(1+y)^{n-k}.$$
This completes the proof.
\end{proof}
If $f(x)$ is $\gamma$-positive, then $\mathcal{A}_{2m+1}(f(x))=f(x^{2m+1})$ may be not alternatingly
$\gamma$-positive. For example, if $f(x)=1+4x+x^2$, then $f(x)=(1+x)^2+2x$ and $$f(x^3)=(1+x)^6-6x(1+x)^4+9x^2(1+x)^2+2x^3,$$
and so $f(x)$ is $\gamma$-positive, while $f(x^3)$ is not alternatingly $\gamma$-positive.
\section{Narayana polynomials}\label{Section04}
\subsection{Identities involving Narayana polynomials}
\hspace*{\parindent}

Let $C_n=\frac{1}{n+1}\binom{2n}{n}$ be the {\it Catalan numbers}.
It is well known that Catalan numbers and the {\it central binomial coefficients} have the following expressions (see~\cite{Chen11,Coker03,Reiner97}):  $$C_n=\sum_{k=0}^{n-1}\frac{1}{n}\binom{n}{k+1}\binom{n}{k},~\binom{2n}{n}=\sum_{k=0}^{n}{\binom{n}{k}}^2.$$
Using generating functions and Lagrange inversion formula,
Coker~\cite{Coker03} derived that
\begin{equation*}\label{Nnx01}
\sum_{k=0}^n\frac{1}{n+1}\binom{n+1}{k+1}\binom{n+1}{k}x^k=\sum_{k=0}^{\lrf{{n}/{2}}}C_k\binom{n}{2k}x^k(1+x)^{n-2k},
\end{equation*}
\begin{equation}\label{Nnx02}
\sum_{k=0}^n\frac{1}{n+1}\binom{n+1}{k+1}\binom{n+1}{k}x^{2k}(1+x)^{2n-2k}=\sum_{k=0}^{n}C_{k+1}\binom{n}{k}x^k(1+x)^{k}.
\end{equation}
Chen, Yan and Yang~\cite{Chen2008} have given combinatorial interpretations of these two identities
based on a bijection between Dyck paths and 2-Motzkin paths, which was first discovered by Delest and Viennot~\cite[p.~179]{Delest84}.
By using generating functions, Riordan~\cite{Riordan68} derived that
\begin{equation}\label{MnxGamma}
\sum_{k=0}^n\binom{n}{k}^2x^{k}=\sum_{k=0}^{\lrf{{n}/{2}}}\binom{n}{2k}\binom{2k}{k}x^k(1+x)^{n-2k}.
\end{equation}
In~\cite{Chen11}, using weighted type $B$ noncrossing partitions as the underlying combinatorial structure, Chen, Wang and Zhao obtained the following identities:
\begin{equation}\label{Nnx03}
\sum_{k=0}^n\binom{n}{k}^2x^{2k}(1+x)^{2n-2k}=\sum_{k=0}^n\binom{n}{k}\binom{2k}{k}x^k(1+x)^k,
\end{equation}

Combining~\eqref{Nnx02},~\eqref{Nnx03} and Theorem~\ref{thm01}, we get the following result.
\begin{theorem}\label{Nn}
For $n\geqslant 0$, one has
\begin{equation}\label{NAnx2}
N(A_n,x^2)=\sum_{k=0}^n C_{k+1}\binom{n}{k}(-x)^k(1+x)^{2n-2k},
\end{equation}
\begin{equation}\label{NBnx2}
N(B_n,x^2)=\sum_{k=0}^{n}\binom{n}{k}\binom{2k}{k}(-x)^k(1+x)^{2n-2k},
\end{equation}
$$\sum_{k=0}^n C_{k+1}\binom{n}{k}x^k=\sum_{k=0}^{\lrf{{n}/{2}}}C_k\binom{n}{2k}x^{2k}(1+2x)^{n-2k},$$
$$\sum_{k=0}^{n}\binom{n}{k}\binom{2k}{k}x^k=\sum_{k=0}^{\lrf{{n}/{2}}}\binom{n}{2k}\binom{2k}{k}x^{2k}(1+2x)^{n-2k},$$
\end{theorem}
\begin{corollary}
For any $n\geqslant 2$, one has
$$N(D_n,x^2)=(1+x)^{2n}+\sum_{i=1}^{n}\left(\binom{n}{i}\binom{2i}{i}-nC_{i-1}\binom{n-2}{i-2}\right)(-x)^i(1+x)^{2n-2i},$$
and hence $N(D_n,x^2)$ is alternatingly $\gamma$-positive.
\end{corollary}
\begin{proof}
The alternating $\gamma$-expansion of $N(D_n,x^2)$ follows from the fact that $$N(D_n,x^2)=N(B_n,x^2)-nx^2N(A_{n-2},x^2).$$
When $i=1$, $\binom{n}{i}\binom{2i}{i}-nC_{i-1}\binom{n-2}{i-2}=2n$, and
for any $2\leqslant i\leqslant n$, we have
$$\frac{\binom{n}{i}\binom{2i}{i}}{nC_{i-1}\binom{n-2}{i-2}}=\frac{2(n-1)(2i-1)}{i(i-1)}=2(n-1)\left(\frac{1}{i}+\frac{1}{i-1}\right)\geqslant 0.$$
\end{proof}
\subsection{The combinatorial proofs of~\eqref{MnxGamma},~\eqref{Nnx03},~\eqref{NAnx2},~\eqref{NBnx2}}\label{SunsectionNa}
\hspace*{\parindent}

A {\it Motzkin path} is a lattice path starting at $(0,0)$, ending at $(n,0)$, and never going below the $x$-axis,
with three possible steps $(1,1),~(1,0)$ and $(1,-1)$. As usual, we use $U,D$ and $H$ to denote an up step $(1,1)$, a down step $(1,-1)$ and a horizontal step $(1,0)$, respectively.
For any $c\in\N$, a {\it $c$-Motzkin path} is a Motzkin path with the horizontal steps can be colored by one of $c$ colors. When $c=0$, there are no horizontal steps and
$0$-Motzkin paths reduce to Dyck paths. When $c=1$, $c$-Motzkin paths reduce to Motzkin paths.
When $c=2$, a horizontal step may be $B$ or $R$, where $B$ and $R$ stand for a blue step and a red step, respectively.
When $c=3$, a horizontal step may be $B,R$ or $G$, where $G$ denotes a green step.
The {\it length} of a lattice path is defined to be the number of steps.
The weight of a path is defined to be the product of the weights of its steps, and the weight of a set of paths equals the sum of weights of its paths.

The following lemma is fundamental.
\begin{lemma}[{\cite{Chen2008}}]\label{Chen2008}
Let $\operatorname{CM}_n$ denote the set of $2$-Motzkin paths of length $n$. Then one has $$\frac{1}{n+1}\binom{n+1}{k+1}\binom{n+1}{k}=\#\{P\in\operatorname{CM}_n:~\UB(P)=k\},$$
where $\UB(P)$ counts $U$ and $B$ steps on $P$. Thus $\#\operatorname{CM}_n=C_{n+1}$.
\end{lemma}

\noindent{\bf Combinatorial proof of the identity~\eqref{NAnx2}:}\\
For any path in $\operatorname{CM}_n$, we assign the weight $x^2$ to each $U$ or $B$ step and the weight $1$ to any other step. By Lemma~\ref{Chen2008},
the left-hand side of~\eqref{NAnx2} equals the weight of $\operatorname{CM}_n$.
It should be noted that the $U$'s and $D$'s must be matched on any path of $\operatorname{CM}_n$.
We use $S(k)$ to denote any subset of $\operatorname{CM}_n$ with $k$ up steps and has the up and down steps in given positions.
Then the weight of $S(k)$ is
$x^{2k}(1+x^2)^{n-2k}$,
since a blue step has the weight $x^2$ and a red step has the weight $1$.

Let $\operatorname{TM}_n$ denote the set of $3$-Motzkin paths of length $n$.
For any path in $\operatorname{TM}_n$, we assign the weight $(-x)$ to each of the $U$, $D$, $B$ and $R$ steps, and the weight $(1+x)^2$ to each $G$ step.
We use $\widehat{S}(k)$ to denote any subset of $\operatorname{TM}_n$ with $k$ up steps and has the up and down steps in given positions.
Since $x^2=(-x)(-x),~1+x^2=(1+x)^2-x-x$, the weight of $\widehat{S}(k)$ equals
$$x^{2k}\left((1+x)^2-x-x\right)^{n-2k}=x^{2k}(1+x^2)^{n-2k},$$
which says that $\widehat{S}(k)$ and $S(k)$ have the same weight.
It remains to show that the weight $\operatorname{TM}_n$ coincides with the right-hand side of~\eqref{NAnx2}.
To construct a path of $\operatorname{TM}_n$ with $n-k$ $G$ steps, we may insert the $G$ steps into $2$-Motzkin paths of $\operatorname{CM}_k$, where all the $U$, $D$, $B$ and $R$ steps have the same weight $(-x)$.
Clearly, there are $\binom{n}{n-k}=\binom{n}{k}$ ways to insert the $G$ steps.
It follows from Lemma~\ref{Chen2008} that $\operatorname{CM}_k=C_{k+1}$. Therefore, the weight of $\operatorname{TM}_n$ equals
$$\sum_{k=0}^n \binom{n}{k}{\left((1+x)^2\right)}^{n-k}C_{k+1}(-x)^k=\sum_{k=0}^n C_{k+1}\binom{n}{k}(-x)^k(1+x)^{2n-2k}.$$
This completes the proof.

For any partition $\lambda=(\lambda_1,\ldots,\lambda_r)\vdash n$, we draw a left-justified array of $n$ cells in the $i$-th row.
This array is called the {\it Young diagram} of $\lambda$. The partition that is represented by such a diagram
is said to be the shape of the diagram.
As usual, we will identify a partition $\lambda$ with its Young diagram.
Let $c$ be a fixed positive integer.
A {\it $c$-colored $2\times n$ Young diagram} is a Young diagram of shape $(n,n)$ such that
each cell may be colored by one of $c$ colors. When $c=1$, we get an ordinary $2\times n$ Young diagram.
When $c=2$, a cell may be colored by black or white.
As illustrated in Figure~\ref{fig:colored}, for any $2$-colored $2\times n$ Young diagram,
we use $U,D,B$ and $N$ to denote a column with a black cell on the top and a white cell at the bottom,
a column with a white cell on the top and a black cell at the bottom, a column with two black cells and a column with two white cells, respectively.
When $c=3$, a cell may be colored by black, white or green.
The weight of a Young diagram is defined to be the product of the weights of its cells, and the weight of a set of Young diagrams is the sum of the weights of its Young diagrams.

\begin{definition}
We use $\operatorname{CY}_n$ to denote the subset of $2$-colored $2\times n$ Young diagrams such that the
top row and bottom row have the same number of black cells.
\end{definition}
It should be noted that the $U$'s and $D$'s must be matched on any Young diagram of $\operatorname{CY}_n$.

\begin{figure}[t]\label{fig:colored}
\setlength{\belowcaptionskip}{0.5cm}
\caption{Four cases of coloring of columns~\label{fig:colored}}
\begin{tabular}{ccccccc}
U: \begin{tabular}{|c|}
  \hline
   \cellcolor{black!60} \\
  \hline
   \\
  \hline
\end{tabular}
& \makebox[20pt]{}&
D: \begin{tabular}{|c|}
  \hline
    \\
  \hline
   \cellcolor{black!60}\\
  \hline
\end{tabular}
&\makebox[20pt]{} &
B: \begin{tabular}{|c|}
  \hline
   \cellcolor{black!60} \\
  \hline
   \cellcolor{black!60}\\
  \hline
\end{tabular}
&\makebox[20pt]{} &
N: \begin{tabular}{|c|}
  \hline
    \\
  \hline
   \\
  \hline
\end{tabular}
\end{tabular}
\end{figure}

\noindent{\bf Combinatorial proof of the identity~\eqref{MnxGamma}:}\\
For any Young diagram in $\operatorname{CY}_n$, we assign the weight $x^{\frac{1}{2}}$ to each black cell and the weight $1$ to each white cell.
Consider the subset of $\operatorname{CY}_n$ consisting of all Young diagrams with exactly $2k$ black cells, i.e., the top row and bottom row both have exactly $k$ black cells.
Since $x=x^{\frac{1}{2}}x^{\frac{1}{2}}$ and there are $\binom{n}{k}$ ways to choose black cells from each row,
the weight of this subset equals $\binom{n}{k}^2x^k$. Thus
the left-hand side of~\eqref{MnxGamma} equals the weight of $\operatorname{CY}_n$. In particular, we have
\begin{equation}\label{Lemma-colored}
\#\operatorname{CY}_n=\binom{2n}{n}
\end{equation}

As illustrated in Figure~\ref{fig:colored}, each column of Young diagrams in $\operatorname{CY}_n$ may be colored with the same color or different colors.
Consider a subset of $\operatorname{CY}_n$ consisting of all Young diagrams having exactly $k$ $U$'s.
Since the $U$'s and $D$'s must be matched, there are $\binom{n}{2k}\binom{2k}{k}$ ways to locate all the $U$'s and $D$'s.
The weight of the other columns is given by $(1+x)^{n-2k}$.
Therefore, the weight of this subset equals
$$\binom{n}{2k}\binom{2k}{k}x^k(1+x)^{n-2k},$$
which is the summand of the right-hand side of~\eqref{MnxGamma}.
The completes the proof.

\begin{definition}
Let $\operatorname{TY}_n$ be the subset of $3$-colored $2\times n$ Young diagram such that
\begin{itemize}
  \item [$(i)$] The top row and bottom row have the same number of black cells;
  \item [$(ii)$] There are only five cases of coloring of columns, in addition to $U,D,B$ and $N$, and a column may be two green cells and we use $G$ to denote it.
\end{itemize}
\end{definition}

\noindent{\bf Combinatorial proof of the identity~\eqref{Nnx03}:}\\
For any Young diagram in $\operatorname{CY}_n$, we first assign the weight $x$ to each black cell and the weight $1+x$ to each white cell.
Note that there are $\binom{n}{k}$ ways to choose black cells in each row.
Then the weight of $\operatorname{CY}_n$ equals
$$\sum_{k=0}^n\left(\binom{n}{k}x^k(1+x)^{n-k}\right)^2=\sum_{k=0}^n\binom{n}{k}^2x^{2k}(1+x)^{2n-2k}.$$
Let $E(k)$ be the subset of $\operatorname{CY}_n$, where each Young diagram has $k$ $U$'s and has the $U$'s and $D$'s in given positions.
Since the $U$'s and $D$'s must be matched, the weight of $E(k)$ equals
$$(x(1+x))^{2k}(x^2+(1+x)^2)^{n-2k}.$$

For any path in $\operatorname{TY}_n$, we assign the weight $x(1+x)$ to each of the $U$, $D$, $B$ and $N$ columns, and the weight $1$ to each $G$ column.
Let $\widehat{E}(k)$ the subset of $\operatorname{TY}_n$, where each Young diagram has $k$ $U$'s and has the $U$'s and $D$'s in given positions.
The weight of $\widehat{E}(k)$ equals
$$(x(1+x))^{2k}(1+x(1+x)+x(1+x))^{n-2k}.$$
Hence $E(k)$ and $\widehat{E}(k)$ has the same weight. For any Young diagram in $\operatorname{TY}_n$ with $n-k$ $G$ columns,
the remains $k$ columns form a new Young diagram in $\operatorname{CY}_k$. It follows from~\eqref{Lemma-colored} that
the weight of the subset of Young diagrams in $\operatorname{TY}_n$ with $n-k$ $G$'s equals
$$\binom{n}{n-k}1^{n-k}\binom{2k}{k}\left(x(1+x)\right)^k=\binom{n}{k}\binom{2k}{k}x^k(1+x)^{k},$$
which is the summand of the right-hand side of~\eqref{Nnx03}. This completes the proof.

\noindent{\bf Combinatorial proof of the identity~\eqref{NBnx2}:}\\
For any Young diagram in $\operatorname{CY}_n$, we reassign the weight $x$ to each black cell and the weight $1$ to each white cell.
In the same way as the combinatorial proof of~\eqref{MnxGamma}, we see that the left-hand side of~\eqref{NBnx2} equals the weight of $\operatorname{CY}_n$.
We use $H(k)$ to denote any subset of $\operatorname{CY}_n$, where each Young diagram with $k$ $U$'s and has the $U$'s and $D$'s in given positions.
The weight of $H(k)$ equals $x^{2k}(1+x^2)^{n-2k}$.

For any Young diagram in $\operatorname{TY}_n$, we assign the weight $(-x)$ to each of the $U$, $D$, $B$ and $N$ columns, and the weight $(1+x)^2$ to each $G$ column.
We use $\widehat{H}(k)$ to denote any subset of $\operatorname{TY}_n$, where each Young diagram with $k$ $U$'s and has the $U$'s and $D$'s in given positions.
Since $x^2=(-x)(-x),~1+x^2=(1+x)^2-x-x$, the weight of $\widehat{H}(k)$ equals
$$x^{2k}\left((1+x)^2-x-x\right)^{n-2k}=x^{2k}(1+x^2)^{n-2k}.$$
Hence $\widehat{H}(k)$ and $H(k)$ have the same weight.
It remains to show that the weight $\operatorname{TY}_n$ coincides with the right-hand side of~\eqref{NBnx2}.
For any Young diagram in $\operatorname{TY}_n$ with $n-k$ $G$ columns,
the remains $k$ columns form a new Young diagram in $\operatorname{CY}_k$. It follows from~\eqref{Lemma-colored} that
the weight of the subset of Young diagrams in $\operatorname{TY}_n$ with $n-k$ $G$'s equals
$$\binom{n}{n-k}\left((1+x)^2\right)^{n-k}\binom{2k}{k}\left(-x\right)^k=\binom{n}{k}\binom{2k}{k}(-x)^k(1+x)^{2n-2k},$$
which is the summand of the right-hand side of~\eqref{NBnx2}. This completes the proof. \qed
\subsection{Hurwitz stability and alternating gamma-positivity}
\hspace*{\parindent}

Let $\rz$ denote the set of real polynomials with only real zeros.
Following~\cite{Wagner96},
we say that a polynomial $p(x)\in\R[x]$ is {\it standard} if its leading coefficient is positive.
Suppose that $p(x),q(x)\in\rz$,
and the zeros of $p(x)$ are $\xi_1\leqslant\cdots\leqslant\xi_n$,
and that those of $q(x)$ are $\theta_1\leqslant\cdots\leqslant\theta_m$.
We say that $p(x)$ {\it interlaces} $q(x)$ if $\deg q(x)=1+\deg p(x)$ and the zeros of
$p(x)$ and $q(x)$ satisfy
$$
\theta_1\leqslant\xi_1\leqslant\theta_2\leqslant\xi_2\leqslant\cdots\leqslant\xi_n
\leqslant\theta_{n+1}.
$$
We say that $p(x)$ {\it alternates left of} $q(x)$ if $\deg p(x)=\deg q(x)$
and the zeros of them satisfy
$$
\xi_1\leqslant\theta_1\leqslant\xi_2\leqslant\theta_2\leqslant\cdots\leqslant\xi_n
\leqslant\theta_n.
$$
Let $\mathbb{C}[x]$ denote the set of all polynomials in $x$ with complex coefficients.
A polynomial $p(x)\in\mathbb{C}[x]$ is {\it Hurwitz stable} if every zero of $p(x)$ is in the open left half plane,
and $p(x)$ is {\it weakly Hurwitz stable} if every zero of $p(x)$
is in the closed left half of the complex plane, see~\cite{Branden0901,Branden0902} for details.
The classical Hermite-Biehler theorem is given as follows.
\begin{HBtheorem}[{\cite[Theorem~3]{Wagner96}}]\label{Hermite}
Let $f(x)=f^E(x^2)+xf^O(x^2)$ be a standard polynomial with real coefficients.
Then $f(x)$ is weakly Hurwitz stable if and only if both $f^E(x)$ and $f^O(x)$ are standard,
have only nonpositive zeros, and $f^O(x)$ interlaces or alternates
left of $f^E(x)$.
Moreover, $f(x)$ is Hurwitz stable if and only if $f(x)$ is weakly Hurwitz stable, $f(0)\neq0$ and
$\gcd(f^E(x),f^O(x))=1$.
\end{HBtheorem}

We now define $$\widehat{N}(n,k)=(n+1)!\frac{1}{n}\binom{n}{k}\binom{n}{k-1},~\widehat{M}(n,k)=n!{\binom{n}{k}}^2.$$
Following~\cite[Lemma~7]{Ma19}, the numbers $\widehat{M}(n,k)$ and $\widehat{N}(n,k)$ satisfy the following recurrences:
\begin{align*}
\widehat{M}(n+1,k)&=(n+1+2k)\widehat{M}(n,k)+(3n+3-2k)\widehat{M}(n,k-1),\\
\widehat{N}(n+1,k)&=(n+2k)\widehat{N}(n,k)+(3n+4-2k)\widehat{N}(n,k-1).
\end{align*}

\begin{lemma}
For $n\geqslant 1$, we have
\begin{equation*}\label{Nnxsum}
\left(\frac{x^2}{1-x^2}D\right)^n\frac{1}{1-x^2}=\frac{(n+1)!x^{n+2}N(A_{n-1},x^2)}{(1-x^2)^{2n+1}},
\end{equation*}
\begin{equation*}\label{Mnxsum}
\left(\frac{x^2}{1-x^2}D\right)^n\frac{x}{1-x^2}=\frac{n!x^{n+1}N(B_n,x^2)}{(1-x^2)^{2n+1}}.
\end{equation*}
Therefore, we have
\begin{equation}\label{MN}
\left(\frac{x^2}{1-x^2}D\right)^n\frac{1}{1-x}=\frac{n!x^{n+1}\left(N(B_n,x^2)+(n+1)xN(A_{n-1},x^2)\right)}{(1-x^2)^{2n+1}}.
\end{equation}
\end{lemma}
\begin{proof}
Note that
\begin{equation*}
\begin{split}
\frac{x^2}{1-x^2}D\frac{1}{1-x^2}&=\frac{2x^3}{(1-x^2)^3},~~
\left(\frac{x^2}{1-x^2}D\right)^2\frac{1}{1-x^2}=\frac{3!x^4(1+x^2)}{(1-x^2)^5},\\
\frac{x^2}{1-x^2}D\frac{x}{1-x^2}&=\frac{x^2(1+x^2)}{(1-x^2)^3},~~
\left(\frac{x^2}{1-x^2}D\right)^2\frac{x}{1-x^2}=\frac{2x^3(1+4x^2+x^4)}{(1-x^2)^5}.
\end{split}
\end{equation*}
Thus the identities hold for $n=1,2$. So we proceed to the inductive step.
Assume that the two identities hold for $n=m$. Then when $n=m+1$, we obtain
\begin{equation*}
\begin{split}
&\left(\frac{x^2}{1-x^2}D\right)\frac{\sum_{k=1}^m\widehat{N}(m,k)x^{2k+m}}{(1-x^2)^{2m+1}}\\
&=\frac{\sum_{k=1}^m(2k+m)\widehat{N}(m,k)x^{2k+m+1}(1-x^2)+2(2m+1)\sum_{k=1}^m\widehat{N}(m,k)x^{2k+m+3}}{(1-x^2)^{2m+3}},\\
&\left(\frac{x^2}{1-x^2}D\right)\frac{\sum_{k=0}^m\widehat{M}(m,k)x^{2k+m+1}}{(1-x^2)^{2m+1}}\\
&=\frac{\sum_{k=0}^m(2k+m+1)\widehat{M}(m,k)x^{2k+m+2}(1-x^2)+2(2m+1)\sum_{k=0}^m\widehat{M}(m,k)x^{2k+m+4}}{(1-x^2)^{2m+3}}.
\end{split}
\end{equation*}
Combining the above two expressions with the recurrences of $\widehat{M}(n,k)$ and $\widehat{N}(n,k)$, we get
$$\left(\frac{x^2}{1-x^2}D\right)^{m+1}\frac{1}{1-x^2}=\frac{x^{m+1}\sum_{k=1}^{m+1}\widehat{N}(m+1,k)x^{2k}}{(1-x^2)^{2m+3}},$$
$$\left(\frac{x^2}{1-x^2}D\right)^{m+1}\frac{x}{1-x^2}=\frac{x^{m+2}\sum_{k=0}^{m+1}\widehat{M}(m+1,k)x^{2k}}{(1-x^2)^{2m+3}},$$
as desired.
Since $$\left(\frac{x^2}{1-x^2}D\right)^n\frac{1}{1-x}=\left(\frac{x^2}{1-x^2}D\right)^{n}\frac{1}{1-x^2}+\left(\frac{x^2}{1-x^2}D\right)^{n}\frac{x}{1-x^2},$$
the proof of~\eqref{MN} follows.
This completes the proof.
\end{proof}

When $n=2$ and $n=3$, the polynomials $N(B_n,x^2)+(n+1)xN(A_{n-1},x^2)$ are
$$1+3x+4x^2+3x^3+x^4=(1+x)^2(1+x+x^2),$$
$$1+4x+9x^2+12x^3+9x^4+4x^5+x^6=(1+x)^2(1+2x+4x^2+2x^3+x^4),$$
respectively. We can now present the following result.
\begin{theorem}\label{thmMN}
For any $n\geqslant 1$, the polynomial $N(B_n,x^2)+(n+1)xN(A_{n-1},x^2)$ is alternatingly $\gamma$-positive, Hurwitz stable, and has a factor $(1+x)^2$.
\end{theorem}
\begin{proof}
Immediate from Theorem~\ref{Nn}, we then get
\begin{equation}\label{NBA}
\begin{split}
&N(B_n,x^2)+(n+1)xN(A_{n-1},x^2)\\&=(1+x)^{2n}+\sum_{k=1}^n\left(\binom{n}{k}\binom{2k}{k}-(n+1)C_k\binom{n-1}{k-1}\right)(-x)^k(1+x)^{2n-2k}.
\end{split}
\end{equation}
For $1\leqslant k\leqslant n$, we see that
$$\frac{\binom{n}{k}\binom{2k}{k}}{(n+1)C_k\binom{n-1}{k-1}}=\frac{\frac{n}{k}\binom{n-1}{k-1}\binom{2k}{k}}{\frac{n+1}{k+1}\binom{2k}{k}\binom{n-1}{k-1}}=\frac{n(k+1)}{(n+1)k}\geqslant 1.$$
In particular, when $k=n$, one has $$\binom{n}{k}\binom{2k}{k}=(n+1)C_k\binom{n-1}{k-1}.$$
So $N(B_n,x^2)+(n+1)xN(A_{n-1},x^2)$ is alternatingly $\gamma$-positive and has a factor $(1+x)^2$.

Recall that
$$N(A_{n},x)=\sum_{k=0}^{n}\frac{1}{n+1}\binom{n+1}{k+1}\binom{n+1}{k}x^k,~N(B_n,x)=\sum_{k=0}^n\binom{n}{k}^2x^k.$$
Then we have
\begin{equation*}
\begin{split}
\frac{\mathrm{d}}{\mathrm{d}x}\left(xN(A_{n},x)\right)&=\sum_{k=0}^n\frac{k+1}{n+1}\binom{n+1}{k+1}\binom{n+1}{k}x^k\\
&=\sum_{k=0}^n\binom{n}{k}\binom{n+1}{k}x^k\\
&=\sum_{k=0}^n\binom{n}{k}^2x^k+\sum_{k=1}^n\binom{n}{k}\binom{n}{k-1}x^k\\
&=N(B_n,x)+nxN(A_{n-1},x),
\end{split}
\end{equation*}
$$\frac{\mathrm{d}}{\mathrm{d}x}\left(N(B_n,x)+nxN(A_{n-1},x)\right)=\sum_{k=1}^n\binom{n}{k}\binom{n+1}{k}kx^{k-1}=n(n+1)N(A_{n-1},x).$$
Since $N(A_{n},x)$ and $N(B_n,x)$ are both real-rooted (see~\cite{Bra06,Liu07}),
the polynomial $N(A_{n-1},x)$ interlaces $N(B_n,x)+nxN(A_{n-1},x)$. Let $r_{n-1}<r_{n-2}<\cdots<r_{1}<0$ be the zeros of $N(A_{n-1},x)$.
The sign of $N(B_{n},r_i)$ is $(-1)^i$ for $i=1,2,\ldots,n-1$. Note that $N(B_{n},x)$ is monic, $N(B_{n},0)=1$ and
$\sgn N(B_{n},-\infty)=(-1)^n$. Hence
$N(B_{n},x)$ has precisely one zero in each of the $n$ intervals $(-\infty,r_{n-1}),(r_{n-1},r_{n-2}),\ldots,(r_2,r_1),(r_1,0)$.
Thus $N(A_{n-1},x)$ interlaces $N(B_{n},x)$. Combining this with the Hermite-Biehler theorem, we get that
$N(B_n,x^2)+(n+1)xN(A_{n-1},x^2)$ are Hurwitz stable. This completes the proof.
\end{proof}

Inductively define the polynomials $L_n(x)$ and $\widehat{L}_n(x)$ by
$$\left(\frac{x^2}{1-x^2}D\right)^n\frac{1}{1-x}=\frac{n!x^{n+1}(1+x)^2L_n(x)}{(1-x^2)^{2n+1}}=\frac{n!x^{n+1}(1+x)\widehat{L}_n(x)}{(1-x^2)^{2n+1}},$$
By induction, it is routine to deduce the following result.
\begin{proposition}\label{LNLN}
For $n\geqslant 1$, we have
\begin{equation*}\label{Lnxrecu}
\begin{split}
nL_n(x)&=(n+2x+(3n-4)x^2)L_{n-1}(x)+x(1-x^2)L_{n-1}'(x),\\
n\widehat{L}_n(x)&=(n+x+(3n-3)x^2)\widehat{L}_{n-1}(x)+x(1-x^2)\widehat{L}_{n-1}'(x),
\end{split}
\end{equation*}
with the initial conditions $L_1(x)=1$ and $\widehat{L}_0(x)=1$.
\end{proposition}
By Theorem~\ref{thmMN}, we immediately get the following result.
\begin{corollary}
Both $L_n(x)$ and $\widehat{L}_n(x)$ are alternatingly $\gamma$-positive and Hurwitz stable.
\end{corollary}
Note that $nL_n(1)=(4n-2)L_{n-1}(1)$. Thus $$2L_n(1)=\widehat{L}_n(1)=\binom{2n}{n}.$$
Below are the polynomials $L_n(x)$ for $n\leqslant 5$:
\begin{equation*}
\begin{split}
L_1(x)&=1,~L_2(x)=1+x+x^2,~L_3(x)=1+2x+4x^2+2x^3+x^4,\\
L_4(x)&=1+3x+9x^2+9x^3+9x^4+3x^5+x^6,\\
L_5(x)&=1+4x+16x^2+24x^3+36x^4+24x^5+16x^6+4x^7+x^8.
\end{split}
\end{equation*}
It should be noted that the sequences $\{L(n,k)\}_{k=0}^{2n-2}$ and $\{\widehat{L}(n,k)\}_{k=0}^{2n-1}$ appear as A088855 in~\cite{Sloane},
which count symmetric Dyck paths by their number of peaks. These sequences have been discussed recently by Cho, Huh and Sohn~\cite[Lemma~3.8]{Cho20}.
Explicitly, we have
$$L(n,k)=\binom{n-1}{\lrc{\frac{k}{2}}}\binom{n-1}{\lrf{\frac{k}{2}}},~\widehat{L}(n,k)=\binom{n}{\lrc{\frac{k}{2}}}\binom{n-1}{\lrf{\frac{k}{2}}}$$
which can be directly verified by using Proposition~\ref{LNLN}.
\section{Identities involving Eulerian polynomials}\label{Section05}
For $\pi\in\msn$, we say that an entry $\pi(i)$ is a {\it left peak} if $\pi(i-1)<\pi(i)>\pi(i+1)$, where $i\in [n-1]$ and $\pi(0)=0$.
Let $\lpk(\pi)$ be the number of left peaks of $\pi$.
The {\it peak polynomials} (also known as interior peak polynomials, see~\cite{Hwang20,Ma121}) and {\it left peak polynomials} are defined by
\begin{align*}
P_n(x)&=\sum_{\pi\in\msn}x^{\pk(\pi)}=\sum_{k=0}^{\lrf{(n-1)/2}}P(n,k)x^k,~
\widehat{P}_n(x)=\sum_{\pi\in\msn}x^{\lpk(\pi)}=\sum_{k=0}^{\lrf{n/2}}\widehat{P}(n,k)x^k,
\end{align*}
respectively. They satisfy the following recurrence relations
\begin{equation}\label{Pnx01}
P_{n+1}(x)=(nx-x+2)P_n(x)+2x(1-x)\frac{\mathrm{d}}{\mathrm{d}x}P_n(x),
\end{equation}
\begin{equation}\label{Pnx02}
\widehat{P}_{n+1}(x)=(nx+1)\widehat{P}_n(x)+2x(1-x)\frac{\mathrm{d}}{\mathrm{d}x}\widehat{P}_n(x),
\end{equation}
with the initial values $P_1(x)=\widehat{P}_1(x)=1$, $P_2(x)=2,~\widehat{P}_2(x)=1+x$, $P_3(x)=4+2x$ and $\widehat{P}_3(x)=1+5x$.
The polynomials $P_n(x)$ and $\widehat{P}_n(x)$ arise often in algebra, combinatorics and other branches of mathematics, see~\cite{Hwang20,Ma121,Petersen06,Stembridge97,Zhuang17} and references therein.
In particular, by using the theory of enriched $P$-partitions, Stembridge~\cite[Remark 4.8]{Stembridge97} found that
\begin{equation}\label{Stembridge}
A_n(x)=\frac{1}{2^{n-1}}\sum_{k=0}^{\lrf{(n-1)/2}}4^kP(n,k)x^k(1+x)^{n-1-2k}.
\end{equation}
It should be noted that by combining~\eqref{Anx-gamma-Foata} and~\eqref{fx2}, we arrive at
\begin{equation*}\label{Anx-altgamma}
A_n(x^2)=\sum_{k=0}^{n-1}\sum_{i=0}^{\lrf{{k}/{2}}}\binom{n-1-2i}{k-2i}2^{k-2i}\gamma_{n,i}(-x)^k(1+x)^{2n-2-2k}.
\end{equation*}
According to~\cite[Observation~3.1.2]{Petersen06}, we have
$$B_n(x)=\sum_{k\geqslant0}4^k\widehat{P}(n,k)x^k(1+x)^{n-2k}.$$
Let $A_n(x)=\sum_{k=0}^{n-1}\Eulerian{n}{k}x^k$ and $B_n(x)=\sum_{k=0}^nB(n,k)x^k$.
We call $\Eulerian{n}{k}$ and $B(n,k)$ the {\it types $A$ and $B$ Eulerian numbers}, respectively.
By Theorem~\ref{thm01}, we get the following two results.
\begin{theorem}\label{thmEulerian01}
\begin{itemize}
  \item [$(i)$] For $n\geqslant 1$, both $A_n(x^2)$ and $B_n(x^2)$ are alternatingly $\gamma$-positive. More precisely,
there exist nonnegative integers $a(n,k)$ and $b(n,k)$ such that
\begin{equation*}
\begin{split}
&\sum_{k=0}^{n-1}\Eulerian{n}{k}x^{2k}=\sum_{k=0}^{n-1}a(n,k)(-x)^k(1+x)^{2n-2-2k},\\
&\sum_{k=0}^{n}B(n,k)x^{2k}=\sum_{k=0}^{n}b(n,k)(-x)^k(1+x)^{2n-2k}.
\end{split}
\end{equation*}
  \item [$(ii)$] There are two identities:
\begin{equation*}
\begin{split}
&\sum_{k=0}^{n-1}\Eulerian{n}{k}x^{2k}(1+x)^{2n-2-2k}=\sum_{k=0}^{n-1}a(n,k)x^k(1+x)^{k},\\
&\sum_{k=0}^{n}B(n,k)x^{2k}(1+x)^{2n-2k}=\sum_{k=0}^{n}b(n,k)x^k(1+x)^{k}.
\end{split}
\end{equation*}
  \item [$(iii)$] Setting $a_n(x)=\sum_{k=0}^{n-1}a(n,k)x^k$ and $b_n(x)=\sum_{k=0}^{n}b(n,k)x^k$, we get
\begin{equation*}\label{anx}
a_n(x)=\frac{1}{2^{n-1}}\sum_{k=0}^{\lrf{(n-1)/2}}4^kP(n,k)x^{2k}(1+2x)^{n-1-2k}=\left(\frac{1+2x}{2}\right)^{n-1}P_n\left(\left(\frac{2x}{1+2x}\right)^2\right),
\end{equation*}
\begin{equation*}\label{bnx}
b_n(x)=\sum_{k=0}^{\lrf{n/2}}4^k\widehat{P}(n,k)x^{2k}(1+2x)^{n-2k}=(1+2x)^n\widehat{P}_n\left(\left(\frac{2x}{1+2x}\right)^2\right).
\end{equation*}
  \item [$(iv)$] We have $a_n(x)=\sum_{i=0}^{n-1}\alpha(n,i)x^i(1+x)^{n-1-i}$,
$b_n(x)=\sum_{i=0}^{n}\beta(n,i)x^i(1+x)^{n-i}$,
\begin{equation}\label{Pnka}
\frac{1}{2^{n-1}}\sum_{k=0}^{\lrf{(n-1)/2}}4^kP(n,k)x^{2k}=\sum_{i=0}^{n-1}\alpha(n,i)(-x)^i(1+x)^{n-1-i},
\end{equation}
\begin{equation}\label{Pnkb}
\sum_{k=0}^{\lrf{n/2}}4^k\widehat{P}(n,k)x^{2k}=\sum_{i=0}^{n}\beta(n,i)(-x)^i(1+x)^{n-i}.
\end{equation}
\end{itemize}
\end{theorem}

We now define
$$\alpha_n(x)=\sum_{i=0}^{n-1}\alpha(n,i)x^i,~~\beta_n(x)=\sum_{i=0}^{n}\beta(n,i)x^i.$$
Setting $y=\frac{-x}{1+x}$ in~\eqref{Pnka} and~\eqref{Pnkb}, we immediately get the following.
\begin{corollary}\label{cor15}
For $n\geqslant 1$, one has
\begin{equation}\label{alp01}
\alpha_n(x)=\frac{1}{2^{n-1}}\sum_{k=0}^{\lrf{(n-1)/2}}4^kP(n,k)x^{2k}(1+x)^{n-1-2k}=\left(\frac{1+x}{2}\right)^{n-1}P_n\left(\left(\frac{2x}{1+x}\right)^2\right),
\end{equation}
\begin{equation*}\label{alp02}
\beta_n(x)=\sum_{k=0}^{\lrf{n/2}}4^k\widehat{P}(n,k)x^{2k}(1+x)^{n-2k}=(1+x)^n\widehat{P}_n\left(\left(\frac{2x}{1+x}\right)^2\right).
\end{equation*}
\end{corollary}
\begin{corollary}
The polynomials $a_n(x),b_n(x),\alpha_n(x)$ and $\beta_n(x)$ satisfy the recurrence relations
\begin{equation}\label{anx-recu}
a_{n+1}(x)=(1+3x-nx)a_n(x)+\frac{1}{2}x(1+4x)\frac{\mathrm{d}}{\mathrm{d}x}a_n(x),
\end{equation}
\begin{equation}\label{bnx-recu}
b_{n+1}(x)=(1+2x-2nx)b_n(x)+x(1+4x)\frac{\mathrm{d}}{\mathrm{d}x}b_n(x),
\end{equation}
\begin{equation}\label{anx1-recu}
\alpha_{n+1}(x)=\left(1+x+\frac{1}{2}(n-1)x(3x-1)\right)\alpha_n(x)+\frac{1}{2}x(1-x)(1+3x)\frac{\mathrm{d}}{\mathrm{d}x}\alpha_n(x),
\end{equation}
\begin{equation}\label{bnx1-recu}
\beta_{n+1}(x)=(1+x-nx+3nx^2)\beta_n(x)+x(1-x)(1+3x)\frac{\mathrm{d}}{\mathrm{d}x}\beta_n(x),
\end{equation}
with the initial conditions $a_1(x)=\alpha_1(x)=b_0(x)=\beta_0(x)=1$. In particular,
$$a_n(-1)=\frac{(-1)^{n-1}}{2^{n-1}}P_n(4),~~b_n(-1)=(-1)^{n}\widehat{P}_n(4), ~~\alpha_n(1)=n!,~~\beta_n(1)=2^nn!.$$
\end{corollary}
\begin{proof}
Differentiation of
\begin{equation*}\label{anx01}
P_n\left(\left(\frac{2x}{1+2x}\right)^2\right)=\left(\frac{2}{1+2x}\right)^{n-1}a_n(x)
\end{equation*}
gives
$$\frac{\mathrm{d}}{\mathrm{d}x}P_n\left(\left(\frac{2x}{1+2x}\right)^2\right)=\frac{2^{n-4}(1+2x)\frac{\mathrm{d}}{\mathrm{d}x}a_n(x)-2^{n-3}(n-1)a_n(x)}{x(1+2x)^{n-3}}.$$
Substituting the above two expressions into~\eqref{Pnx01} and simplifying, we get~\eqref{anx-recu}.
Differentiation of
\begin{equation*}\label{anx01}
\widehat{P}_n\left(\left(\frac{2x}{1+2x}\right)^2\right)=\frac{b_n(x)}{(1+2x)^n},
\end{equation*}
gives
$$\frac{\mathrm{d}}{\mathrm{d}x}\widehat{P}_n\left(\left(\frac{2x}{1+2x}\right)^2\right)=\frac{(1+2x)\frac{\mathrm{d}}{\mathrm{d}x}b_n(x)-2nb_n(x)}{8x(1+2x)^{n-2}}.$$
Substituting the above two expressions into~\eqref{Pnx02} and simplifying, we obtain~\eqref{bnx-recu}.

Differentiation of
\begin{equation*}\label{anx01}
P_n\left(\left(\frac{2x}{1+x}\right)^2\right)=\left(\frac{2}{1+x}\right)^{n-1}\alpha_n(x)
\end{equation*}
gives
$$\frac{\mathrm{d}}{\mathrm{d}x}P_n\left(\left(\frac{2x}{1+x}\right)^2\right)=\frac{2^{n-4}(1+x)\frac{\mathrm{d}}{\mathrm{d}x}\alpha_n(x)-2^{n-4}(n-1)\alpha_n(x)}{x(1+x)^{n-3}}.$$
Substituting the above two expressions into~\eqref{Pnx01} and simplifying, we get~\eqref{anx1-recu}.
Differentiation of
\begin{equation*}\label{anx01}
\widehat{P}_n\left(\left(\frac{2x}{1+x}\right)^2\right)=\frac{\beta_n(x)}{(1+x)^n},
\end{equation*}
gives
$$\frac{\mathrm{d}}{\mathrm{d}x}\widehat{P}_n\left(\left(\frac{2x}{1+x}\right)^2\right)=\frac{(1+x)\frac{\mathrm{d}}{\mathrm{d}x}\beta_n(x)-n\beta_n(x)}{8x(1+x)^{n-2}}.$$
Substituting the above two expressions into~\eqref{Pnx02} and simplifying, we arrive at~\eqref{bnx1-recu}.
\end{proof}

For convenience, we list the first few $a_n(x)$'s,~$b_n(x)$'s,$\alpha_n(x)$'s and $\beta_n(x)$'s:
$$a_1(x)=1,~a_2(x)=1+2x,~a_3(x)=1+4x+6x^2,~a_4(x)=1+6x+20x^2+24x^3;$$
$$b_1(x)=1+2x,~b_2(x)=1+4x+8x^2,~b_3(x)=1+6x+32x^2+48x^3;$$
$$\alpha_1(x)=1,~\alpha_2(x)=1+x,~\alpha_3(x)=1+2x+3x^2,~\alpha_4(x)=1+3x+11x^2+9x^3;$$
$$\beta_1(x)=1+x,~\beta_2(x)=1+2x+5x^2,~b_3(x)=1+3x+23x^2+21x^3.$$

Recall that $\alpha_n(1)=n!$. We shall provide a combinatorial interpretation for $\alpha_n(x)$.
Define $$\widehat{\alpha}_n(x)=x^{n-1}\alpha_n\left(\frac{1}{x}\right)$$ for $n\geqslant 1$, and $\widehat{\alpha}_0(x)=1$.
Combining~\eqref{Anx-gamma-Foata},~\eqref{Stembridge} and~\eqref{alp01}, we see that
\begin{align*}
\widehat{\alpha}_n(x)&=\left(\frac{1+x}{2}\right)^{n-1}P_n\left(\left(\frac{2}{1+x}\right)^2\right)\\
&=\frac{1}{2^{n-1}}\sum_{k=0}^{\lrf{(n-1)/2}}4^kP(n,k)(1+x)^{n-1-2k}\\
&=\sum_{k=0}^{\lrf{(n-1)/2}}\gamma_{n,k}(1+x)^{n-1-2k},
\end{align*}
where $\gamma_{n,k}=\#\{\pi\in\msn:~\pk(\pi)=k,~\ddes(\pi)=0\}$.
By using the $\MFS$-action defined by~\ref{MFS}, one can immediately get that
$$\widehat{\alpha}_n(x)=\sum_{\pi\in\msn}x^{\dasc(\pi)},$$
since each double ascent of $\pi$ can be transformed to a double descent.
It is well known~\cite[\text{A008303}]{Sloane} that the exponential generating function of peak polynomials is given as follows:
\begin{equation}\label{alp04}
  P(x;z):=\sum_{n=1}^{\infty}P_n(x)\frac{z^n}{n!}=\frac{\sinh(z\sqrt{1-x})}{\sqrt{1-x}\cosh(z\sqrt{1-x})-\sinh(z\sqrt{1-x})},
\end{equation}
Note that
\begin{equation}\label{alp05}
\widehat{\alpha}(x;z):=\sum_{n=0}^{\infty}\widehat{\alpha}_n(x)\frac{z^n}{n!}=1+\frac{2}{1+x}P\left(\left(\frac{2}{1+x}\right)^2;\frac{(1+x)z}{2}\right).
\end{equation}
Set $u=\sqrt{(x+3)(x-1)}$.
Combining~\eqref{alp04} and~\eqref{alp05}, it is routine to verify that
$$\widehat{\alpha}(x;z)=\frac{u\cosh\left(\frac{1}{2}uz\right)+(1-x)\sinh\left(\frac{1}{2}uz\right)}{u\cosh\left(\frac{1}{2}uz\right)-(1+x)\sinh\left(\frac{1}{2}uz\right)},$$
which was also recently studied by Zhuang~\cite[Theorem~13]{Zhuang17}.
In conclusion, we can now restate Corollary~\ref{cor15}.
\begin{proposition}\label{thm08}
For $n\geqslant 1$, one has $$\alpha_n(x)=\sum_{\pi\in\msn}x^{n-1-\dasc(\pi)}=\sum_{\pi\in\msn}x^{\pk(\pi)+\des(\pi)},$$
$$\beta_n(x)=\sum_{\pi\in\msn}(2x)^{2\lpk(\pi)}(1+x)^{n-2\lpk(\pi)}.$$
\end{proposition}
\section{Gamma-positivity and alternating semi-gamma-positivity}\label{Section06}
Following~\cite[Definition~15]{Ma2001}, if $f(x)$ has the expansion
\begin{equation}\label{fx3}
f(x)=(1+x)^\nu\sum_{k=0}^{n}\lambda_kx^k(1+x^2)^{n-k},
\end{equation}
and $\lambda_k\geqslant 0$ for all $0\leqslant k\leqslant n$, then we say that $f(x)$ is {\it semi-$\gamma$-positive}, where $\nu=0$ or $\nu=1$.
Rewriting $f(x)$ in the form $f(x)=(1+x)^\nu \left(f_1(x^2)+xf_2(x^2)\right)$,
we see that $f(x)$ is semi-$\gamma$-positive if and only if $f_1(x)$ and $f_2(x)$ are both $\gamma$-positive polynomials (see~\cite[Proposition~16]{Ma2001}).
The $\gamma$-positivity implies semi-$\gamma$-positivity, but not vice versa.

Assume that $f(x)$ has the expansion~\eqref{fx3}.
If $f(x)$ is semi-$\gamma$-positive, then it follows from~\eqref{fx2} that there exist nonnegative integers $\xi_k$ and $\zeta_k$ such that
\begin{equation*}
\begin{split}
\frac{f(x)}{(1+x)^\nu}&=\sum_{k=0}^{\lrf{n/2}}\lambda_{2k}x^{2k}(1+x^2)^{n-2k}+x\sum_{k=0}^{\lrf{n/2}}\lambda_{2k+1}x^{2k}(1+x^2)^{n-2k-1}\\
&=\sum_{k=0}^{n}\xi_k(-x)^k(1+x)^{2n-2k}+x\sum_{k=0}^{n-1}\zeta_k(-x)^k(1+x)^{2n-2-2k}.
\end{split}
\end{equation*}

Thus $$f(x)=\sum_{k=0}^{n}(\xi_k-\zeta_{k-1})(-x)^k(1+x)^{2n-2k+\nu},$$
where $\zeta_{-1}=0$.
Hence $f(x)$ may be not alternatingly $\gamma$-positive, but the following result holds.
\begin{lemma}\label{semi}
If $f(x)$ is semi-$\gamma$-positive, then it can be written as a linear combination of two alternatingly $\gamma$-positive polynomials, i.e.,
there exist nonnegative integers $\xi_k$ and $\zeta_k$ such that
\begin{equation}\label{fxbi}
f(x)=\sum_{k=0}^{n}\xi_k(-x)^k(1+x)^{2n-2k+\nu}+x\sum_{k=0}^{n-1}\zeta_k(-x)^k(1+x)^{2n-2k-2+\nu}.
\end{equation}
\end{lemma}
We say that $f(x)$ is {\it alternatingly semi-$\gamma$-positive} if it can be written as the linear combination~\eqref{fxbi}.
Since $\gamma$-positivity implies semi-$\gamma$-positivity, we get the following result.
\begin{theorem}\label{Fnx}
If $f(x)$ is a $\gamma$-positive polynomial, then it is also alternatingly semi-$\gamma$-positive.
\end{theorem}
Here we provide two examples, one is $\gamma$-positive, the other is not $\gamma$-positive.
\begin{example}
Consider $A_6(x)=1+57x+302x^2+302x^3+57x^4+x^5$, which is the Eulerian polynomial for the symmetric group $\ms_6$. Then
\begin{equation*}
\begin{split}
A_6(x)&=(1+x)(1+56x+246x^2+56x^3+x^4)\\
&=(1+x)^5-4x(1+x)^3+248x^2(1+x)+x\left(56(1+x)^3-112x(1+x)\right).
\end{split}
\end{equation*}
\end{example}
\begin{example}
Consider $f(x)=1+7x+29x^2+31x^3+29x^4+7x^5+x^6$. We have
\begin{equation*}
\begin{split}
f(x)&=(1+x)^6+x(1+x)^4+10x^2(1+x)^2-15x^3\\
&=(1+x^2)^3+7x(1+x^2)^2+26x^2(1+x^2)+17x^3\\
&=(1+x)^6-6x(1+x)^4+38x^2(1+x)^2-60x^3+x\left(7(1+x)^4-28x(1+x)^2+45x^2\right).
\end{split}
\end{equation*}
Hence $f(x)$ is not $\gamma$-positive, but it is semi-$\gamma$-positive and alternatingly semi-$\gamma$-positive.
\end{example}

Stirling permutations were introduced by Gessel and Stanley~\cite{Gessel78}.
A {\it Stirling permutation} of order $n$ is a permutation of the multiset $\{1,1,2,2,\ldots,n,n\}$ such that
for each $i$, $1\leqslant i\leqslant n$, all entries between the two occurrences of $i$ are larger than $i$.
Denote by $\mqn$ the set of {\it Stirling permutations} of order $n$.
Let $\sigma=\sigma_1\sigma_2\cdots\sigma_{2n}\in\mqn$.
An occurrence of an {\it ascent-plateau} of $\sigma\in\mqn$ is an index $i$ such that $\sigma_{i-1}<\sigma_{i}=\sigma_{i+1}$, where $i\in\{2,3,\ldots,2n-1\}$.
An occurrence of a {\it left ascent-plateau} is an index $i$ such that $\sigma_{i-1}<\sigma_{i}=\sigma_{i+1}$, where $i\in\{1,2,\ldots,2n-1\}$ and $\sigma_0=0$.
Let $\ap(\sigma)$ (resp.~$\lap(\sigma)$) be the number of ascent-plateaus (resp.~left ascent-plateaus) of $\sigma$, see~\cite{Ma2001,Ma22} for details.

The {\it flag ascent-plateau polynomials} are defined by $$F_n(x)=\sum_{\sigma\in\mqn}x^{\fap(\sigma)},~F_0(x)=1,$$
where $\fap(\sigma)=\ap(\sigma)+\lap(\sigma)$.
They satisfy the recurrence relation
\begin{equation*}
F_{n+1}(x)=(x+2nx^2)F_n(x)+x(1-x^2)\frac{\mathrm{d}}{\mathrm{d}x}F_n(x).
\end{equation*}
Combining~\cite[Corollary~20]{Ma2001} and~\cite[Theorem~4.13]{Zhuang1702}, we get
$$2x(1+x)^{n-1}A_n(x)=\sum_{k=0}^n\binom{n}{k}F_k(x)F_{n-k}(x)$$
for $n\geqslant 1$, where $A_n(x)$ are the Eulerian polynomials.
Below are $F_n(x)$ for $n\leqslant 5$:
\begin{equation*}
\begin{split}
  F_1(x)&=x,~F_2(x)=x+x^2+x^3,~~F_3(x)=x+3x^2+7x^3+3x^4+x^5,\\
  F_4(x)&=x+7x^2+29x^3+31x^4+29x^5+7x^6+x^7.
\end{split}
\end{equation*}
According to~\cite[Proposition~18,~Theorem~19]{Ma2001},
the polynomials $F_n(x)$ are not $\gamma$-positive, but they are semi-$\gamma$-positive. By Lemma~\ref{semi}, we immediately get the following result.
\begin{proposition}
The flag ascent-plateau polynomials are alternatingly semi-$\gamma$-positive.
\end{proposition}
\section{Two conjectures}\label{Section07}
We shall present two conjectures for future research.
We now recall an elementary result.
\begin{proposition}[{\cite{Beck2010,Branden18}}]
Let $f(x)$ be a polynomial of degree $n$.
There is a unique decomposition $f(x)= a(x)+xb(x)$, where
\begin{equation*}
a(x)=\frac{f(x)-x^{n+1}f(1/x)}{1-x},~b(x)=\frac{x^nf(1/x)-f(x)}{1-x}.
\end{equation*}
\end{proposition}
Clearly, $a(x)$ and $b(x)$ are symmetric polynomials.
The ordered pair of polynomials $(a(x),b(x))$ is called the {\it symmetric decomposition} of $f(x)$.
We say that $f(x)$ is {\it bi-$\gamma$-positive} (resp.~{\it alternatingly bi-$\gamma$-positive})
if $a(x)$ and $b(x)$ are both $\gamma$-positive (resp. alternatingly $\gamma$-positive).
\subsection{On the Boros-Moll polynomials}
\hspace*{\parindent}

Boros and Moll~\cite{Boros9901,Boros9902} have shown that for any $x\geqslant -1$ and $m\in\N$, there exists a sequence of polynomials $M_m(x)$ satisfying
$$\int_0^{\infty}\frac{1}{(1+2xy^2+y^4)^{m+1}}dy=\frac{\pi}{2^{m+3/2}(x+1)^{m+1/2}}M_m(x),$$
where $M_m(x)$ are called the {\it Boros-Moll polynomials}.
Explicitly,
$M_m(x)=\sum_{i=0}^md_i(m)x^i$, where
$$d_i(m)=2^{-2m}\sum_{k=i}^m2^k\binom{2m-2k}{m-k}\binom{m+k}{k}\binom{k}{i}.$$
For any $m\geqslant 1$, Boros and Moll~\cite{Boros9902} showed that the polynomial $M_m(x)$ is unimodal and the model of it appears in the middle.
For example, $$M_5(x)=\frac{4389}{256}+\frac{8589}{128}x+\frac{7161}{64}x^2+\frac{777}{8}x^3+\frac{693}{16}x^4+\frac{63}{8}x^5.$$
Using the RISC package MultiSum, Kauers and Paule~\cite[Eq.~(6)]{Kausers07} found that for $0\leqslant i\leqslant m+1$,
the numbers $d_i(m)$ satisfy the recurrence relation
\begin{equation}\label{recu01}
2(m+1)d_i(m+1)=2(m+i)d_{i-1}(m)+(4m+2i+3)d_i(m).
\end{equation}
The Boros-Moll polynomials have been extensively studied, see~\cite{Chen2009} and references therein.

We now define $$Q_m(x)=2^mm!x^mM_m\left(\frac{1}{x}\right)=\sum_{i=0}^mc_i(m)x^i,~Q_0(x)=1.$$
Then $c_i(m)=2^mm!d_{m-i}(m)$. It follows from~\eqref{recu01} that the numbers $c_i(m)$ satisfy the recurrence
\begin{equation}\label{cim}
c_i(m+1)=(4m-2i+2)c_i(m)+(6m-2i+5)c_{i-1}(m),
\end{equation}
with the initial conditions $c_0(0)=1$ and $c_i(0)=0$ for all $i\neq 0$.
Multiplying both sides of~\eqref{cim} by $x^i$ and summing over $i$, we obtain
$$Q_{m+1}(x)=(2m+1)(2+3x)Q_m(x)-2x(1+x)\frac{\mathrm{d}}{\mathrm{d}x}Q_m(x).$$
Below are the symmetric decompositions of the polynomials $Q_m(x)$ for $m\leqslant 4$:
\begin{equation*}
\begin{split}
Q_1(x)&=2+3x=2(1+x)+x,~~Q_2(x)=3(4+7x+4x^2)+9x(1+x),\\
Q_3(x)&=120+420x+516x^2+231x^3\\
=&4(40+103x+103x^2+40 x^3)+3x(37+69x+37 x^2),\\
Q_4(x)&=1680+7560x+13140x^2+10620x^3+3465x^4\\
&=105(16+55x+79x^2+55x^3+16 x^4)+255x(1+x)(7+12x+7x^2),\\
Q_5(x)&=30240+166320x+372960x^2+429660x^3+257670x^4+65835x^5\\
&=315(96+415x+781x^2+781x^3+415x^4+96x^5)+\\
&315x(113+403x+583x^2+403x^3+113x^4).
\end{split}
\end{equation*}
Based on empirical evidence, we propose the following.
\begin{conjecture}
The polynomial $Q_m(x)$ has a decomposition $Q_m(x)=a_m(x)+xb_m(x)$ for any $m\geqslant 1$,
where $a_m(x)$ and $b_m(x)$ are both symmetric and unimodal polynomials. Moreover, $Q_m(x)$ is alternatingly bi-$\gamma$-positive.
So $M_m(x)$ is also alternatingly bi-$\gamma$-positive.
\end{conjecture}
\subsection{On the enumerators of permutations by descents and excedances}
\hspace*{\parindent}

Let $\pi\in\msn$. The {\it major index} of $\pi$ is the defined by
$$\maj(\pi)=\sum_{\pi(i)>\pi(i+1)}i.$$
An {\it excedance} of $\pi$ is an index $i\in [n-1]$ such that $\pi(i)>i$.
Let $\exc(\pi)$ denote the number of excedacnes of $\pi$.
In~\cite{Shareshian10}, Shareshian and Wachs studied the trivariate Eulerian polynomials
$$A_n^{\maj,\des,\exc}(q,p,q^{-1}t)=\sum_{\pi\in\msn}q^{\maj(\pi)-\exc(\pi)}p^{\des(\pi)}t^{\exc(\pi)}.$$
They noted that these polynomials are not $t$-symmetric and studied the unimodality of several associated polynomials, see~\cite[p.~2951]{Shareshian10} for details.

We now consider the following bivariate Eulerian polynomials
$$A_n(s,t)=\sum_{\pi\in\msn}s^{\des(\pi)}t^{\exc(\pi)}.$$
According to~\cite[Eq~(1.15),~Eq.~(1.18)]{Foata08}, we have
\begin{equation}\label{Anst}
\sum_{n\geqslant 0}A_n(s,t)\frac{u^n}{(1-s)^{n+1}}=\sum_{r\geqslant 0}s^r\frac{1-t}{(1-u)^{r+1}(1-ut)^{-r}-t(1-u)}.
\end{equation}
Below are the polynomials $A_n(s,t)$ for $1\leqslant n\leqslant 5$:
\begin{equation*}
\begin{split}
A_1(s,t)&=1,~A_2(s,t)=1+st=1+t+(s-1)t,\\
A_3(s,t)&=1+(3s+s^2)t+st^2={(1+(1+s)^2t+t^2)}+(s-1)t(1+t),\\
A_4(s,t)&=1+(6s+5s^2)t+(4s+6s^2+s^3)t^2+st^3\\
=&(1+t){(1+5s(1+s)t+t^2)}+(s-1)t{(1+(1+s)^2t+t^2)},\\
A_5(s,t)&=1+(10s+15s^2+s^3)t+(10s+36s^2+19s^3+s^4)t^2+\\
&(5s+15s^2+6s^3)t^3+st^4\\
=&(1+(1+9s+15s^2+s^3)t+(1+14s+36s^2+14s^3+s^4)t^2+\\
&(1+9s+15s^2+s^3)t^3+t^4)+(s-1)t(1+t)(1+5s(1+s)t+t^2).
\end{split}
\end{equation*}
We end this paper by giving the following conjecture.
\begin{conjecture}
Let $s\geqslant 1$ be a given real number.
For any $n\geqslant2$, the polynomial $A_n(s,t)$ has the following symmetric decomposition:
\begin{equation*}\label{Anstdecom}
A_n(s,t)=a_n(s,t)+(s-1)ta_{n-1}(s,t),
\end{equation*}
where $a_n(s,t)$ is $\gamma$-positive. Thus $A_n(s,t)$ is unimodal with mode in the middle.
\end{conjecture}

\end{document}